\documentclass[a4paper, 11pt]{article}

\usepackage{amsmath} \usepackage{amsthm}
\usepackage{amsfonts}\usepackage{amssymb}
\usepackage{graphicx}
\usepackage{color}

\newtheorem{theorem}{Theorem} \newtheorem{prop}[theorem]{Proposition}
\newtheorem{lemma}[theorem]{Lemma} \newtheorem{defi}{Definition}
 \newtheorem{rmk}[theorem]{Remark}

\newcommand{\eps}{\varepsilon} 

\newcommand{\indic}[1]{\mathbf{1}_{\{#1\}}}

\DeclareMathOperator\var{var}
\DeclareMathOperator\cov{{cov}}
\newcommand{\half}{\frac1{2}}

\newcommand{\keywords}{\textbf{Keywords}\ }

\def\mn{\medskip\noindent}
\def\bn{\bigskip\noindent}
\def\beq{\begin{equation}}
\def\eeq{\end{equation}}
\def\beqa{\begin{eqnarray}}
\def\eeqa{\end{eqnarray}}
\def\beqax{\begin{eqnarray*}}
\def\eeqax{\end{eqnarray*}}
\def\sqz{\kern -0.2em}

\def\cA{\mathcal A}

\def\cC{\mathcal C}

\def\cS{\mathcal S}

\def\E{\mathbb{E}}
\def\P{\mathbb{P}}
\def\R{\mathbb{R}}

\def\gg{|\Gamma|}
\def\g{\Gamma}
\def\nc{\bar N}
\begin{document}

\title{Emergence of giant cycles and slowdown transition in random transpositions and $k$-cycles}

\author{Nathana\"el Berestycki$^{1}$}
\date{April 2010}
\maketitle

\centerline{\textbf{Abstract}}

\mn Consider the random walk on the permutation group
obtained when the step distribution is uniform on a given conjugacy class. It is shown that there is a critical time at which two phase transitions occur simultaneously. On the one hand, the random walk slows down abruptly (i.e., the acceleration drops from 0 to $-\infty$ at this time as $n \to \infty$). On the other hand, the largest cycle size changes from microscopic to giant. The proof of this last result is both considerably simpler and more general than in a previous result of Oded Schramm (2005) for random transpositions. It turns out that in the case of random $k$-cycles, this critical time is proportional to $1/[k(k-1)]$, whereas the mixing time is known to be proportional to $1/k$.

\bn  \noindent \keywords{random transpositions, random $k$-cycles, random permutations, cycle percolation, coalescence-fragmentation, random hypergraphs, conjugacy class, mixing time}

\vfill
\bn 1. Statistical Laboratory, Cambridge University. Wilberforce Rd., Cambridge CB3 0WB.

\newpage

\bn
\section{Introduction}

\subsection{Basic result}

Let $n\ge 1$ and let $\mathcal{S}_n$ be the group of permutations of $\{1,\ldots, n \}$. Consider the random walk on $\mathcal{S}_n$ obtained by performing random transpositions in continuous time, at rate 1. That is, let $\tau_1, \ldots$ be a sequence of i.i.d. uniformly chosen transpositions among the $n(n-1)/2$ possible transpositions of the set $V=\{1,\ldots, n\}$, and for all $t\ge 0$, set
$$
\sigma_t = \tau_1 \cdot \ldots \cdot \tau_{N_t}
$$
where $(N_t,t\ge 0)$ is an independent Poisson process with rate 1. It is well-known that the permutation $\sigma_t$ is approximately a uniform random permutation (in the sense of total variation distance) after time $(1/2) n \log n$ (see \cite{d-sh}). In particular, this means that at this time, most points belong to cycles which are of macroscopic size $O(n)$, while initially, in the permutation $\sigma_0$ which is the identity permutation, every cycle is microscopic (being of size 1). How long does it take for macroscopic cycles to emerge? Oded Schramm, in a remarkable paper \cite{schramm}, proved that the first giant cycles appear at time $n/2$. More precisely, answering a conjecture of David Aldous stated in \cite{bd}, he was able to prove that if $t=cn$ with $c>1/2$, then there exists a (random) set $W\subset \{1,\ldots, n\}$ satisfying $\sigma_t(W) = W$, such that $|W| \sim \theta n$ where $0<\theta = \theta(c)<1$, and furthermore,
the cycle lengths of $\sigma_t|_W$, rescaled by $\theta n$, converges in the sense of finite-dimensional distributions towards a Poisson-Dirichlet random variable. (The Poisson-Dirichlet distribution describes the limiting cycle distribution of a uniform random permutation and will be described in more details below). In particular, this implies that $\sigma_t$ contains giant cycles with high probability. On the other hand it is furthermore easy to see that no macroscopic cycle can occur if $c<1/2$.
His proof is separated into two main steps. The first step consists in showing that giant cycles do emerge prior to time $cn$ when $c>1/2$. The second step is a beautiful coupling argument which shows that once giant cycles exist they must quickly come close to equilibrium, thereby proving Aldous' conjecture. Of these two steps, the first is arguably the most technically involved. Our main purpose in this paper is to give an elementary and transparent new proof of this fact.  Let $\Lambda(t)$ denote the size of the largest cycle of $\sigma_t$. For $\delta>0$, define
\begin{equation}\label{taudelta}
\tau_\delta =\inf\{t\ge 0: \Lambda(t) > \delta n\} .
\end{equation}

\begin{theorem}
\label{T:perctransp}
For any $c>1/2$ then $\tau_\delta <cn$ with high probability, where $$\delta = \frac{\theta(c)^2}8>0.$$
\end{theorem}

This proof is completely elementary and in particular requires almost no estimate. As a consequence, it is fairly robust and it can be hoped that it extends to further models. We illustrate this by applying it to more general random walks on $\mathcal{S}_n$, whose step distribution is uniform on a given conjugacy class of the permutation group (definitions will be recalled below). We show that the emergence of giant cycles coincides with a phase transition in the speed of the random walk, as measured by the derivative of the distance (with respect to the graph metric) between the position of the random walk at time $t$, and its starting point. This phase transition in the speed is the analogue of the phase transition described in \cite{bd} for random transpositions.

\medskip We mention that Theorem \ref{T:perctransp} is the mean-field analogue of a question arising in statistical mechanics in the study of Bose condensation and the quantum ferromagnetic Heisenberg model (see T\`oth \cite{toth}). Very few rigorous results are known about this model on graphs with non-trivial geometry, with the exception of the work of Angel \cite{angel} for the case of a $d$-regular tree with $d$ sufficiently large. We believe that the proof of Theorem \ref{T:perctransp} proposed here opens up the challenging possibility to prove analogous results on graphs that are ``sufficiently high-dimensional" such as a high-dimensional hypercube, for which the percolation picture has recently started to emerge: see, e.g., Borgs et al. \cite{cube}.

\subsection{Random walks based on conjugacy classes.}
\label{S:conj}
Fix a number $k\ge 2$, and call an element $\gamma \in \mathcal{S}_n$ a $k$-cycle, or a cyclic permutation of length $k$, if there exist pairwise distinct
elements $x_1,\ldots,x_k \in \{1,\dots, n\}$ such that
$\gamma(x)=x_{i+1}$ if  $x=x_i$ (where $1\le i\le k$ and $x_{k+1} := x_1$) and
$ \gamma(x) = x $ otherwise.
Thus for $k=2$, a 2-cycle is simply a transposition. If $\sigma$ is a permutation then $\sigma$ can be decomposed into a product of cyclic permutations $\sigma = \gamma_1 \cdot \ldots \cdot \gamma_r$ where $\cdot$ stands for the composition of permutations. (This decomposition being unique up to the order of the terms). A conjugacy class $\Gamma \subset \cS_n$ is any set that is invariant by conjugacy $\sigma \mapsto \pi^{-1} \sigma \pi$, for all $\pi \in \cS_n$. It easily seen that a conjugacy class of $\cS_n$ is exactly a set of permutations having a given cycle structure, say $(k_2, \ldots, k_J)$, i.e., consisting of $k_2$ cycles of size 2, $\ldots$, $k_J$ cycles of size $J$ in their cycle decomposition (and a number of fixed points which does not need to be explicitly stated). Note that if $\Gamma$ is a fixed conjugacy class of $\cS_n$, and $m>n$, $\g$ can also be considered a conjugacy class of $\cS_m$ by simply adding $m-n$ fixed points to any permutation $\sigma \in \g$.

\medskip Let $\g$ be a fixed conjugacy class, and consider the random walk in continuous time on $\mathcal{S}_n$
where the step distribution is uniform on $\g$. That is, let $(\gamma_i,i\ge 1)$ be an i.i.d. sequence of elements uniformly distributed on $\g$, and let $(N_t, t\ge 0)$ be an independent rate 1 Poisson process. Define a random
process:
\begin{equation}
\sigma_t:=\gamma_1\cdot \ldots \cdot \gamma_{N_t}, \ \ t\ge 0,
\end{equation}
where $\cdot$ stands for the composition of two permutations. Thus the case where $\g$ consists only of transpositions (i.e. $k_2 =1$ and $k_j =0$ if $j\ge 2$) corresponds to the familiar random process on $\mathcal{S}_n$ obtained by performing random transpositions in continuous time, and the case where $\g$ contains only one nontrivial cycle of size $k\ge 2$ will be referred to as the random $k$-cycles random walk. The process $(\sigma_t, t \ge 0)$ may conveniently be viewed as a random walk on $G_n$, the Cayley graph of $\mathcal{S}_n$ generated by $\g$. Note that if $\gg = \sum_{j=2}^J jk_j$ is even, the graph $G_n$ is connected but it is not when $\gg$ is odd: indeed, in that case, the product of random $p$-cycles must be an even permutation, and thus $\sigma_t$ is then a random walk on the alternate group $\mathcal{A}_n$ of even permutations. This fact will be of no relevance in what follows.

\medskip In this paper we study the pre-equilibrium behaviour of such a random walk. Our main result in this paper for this process is that there is a phase transition which occurs at time $t_c n$, where
\begin{equation}\label{tc}
t_c = \left(\sum_{j=2}^J j(j-1) k_j\right)^{-1}.
\end{equation}
This transition concerns two distinct features of the walk. On the one hand, giant cycles emerge at time $t_c n$ precisely, as in Theorem \ref{T:perctransp}. On the other hand, the speed of the walk changes dramatically at this time, dropping below 1 in a non-differentiable way. We start with the emergence of giant cycles, which is analogue to Theorem \ref{T:perctransp}. Recall the definition of $\tau_\delta$ in \eqref{taudelta}.

\begin{theorem}\label{T:perc2}
Let $t<t_c$. Then there exists $\beta>0$ such that no cycle is greater than $\beta \log n$ with high probability. On the other hand for any $t>t_c$ there exists $\delta>0$ such that $\tau_\delta < tn$ with high probability.
\end{theorem}

We now state our result for the speed. Denote by $d(x,y)$ the graph distance between two vertices $x,y \in \mathcal{S}_n$, and for $t\ge 0$, let $$d(t)=d(o,\sigma_t).$$ where $o$ is the identity permutation of $\mathcal{S}_n$. Recall that a sequence of random functions $X_n(t)$ converge uniformly on compact sets of $S \subset \R$ in probability (\emph{u.c.p.} for short) towards a random function $X(t)$ if $\P(\sup_{t \in S, t \le T} | X_n(t) - X(t)| > \eps) \to 0$ as $n \to\infty$ for all $\eps>0$ and $T>0$.

\begin{theorem} \label{3-cycle} Fix a constant integer $J\ge 2$ and constant nonnegative integers $k_2, \ldots, k_J$, and consider the conjugacy class $\g$ of $\cS_n$ defined by $(k_2, \ldots, k_J)$. Let $t_c$ be as in \eqref{tc}, and fix $t>0$. Then there exists a compact interval $I \subset (t_c, \infty)$, and a nonrandom function $\varphi(t)$ satisfying $\varphi(t) = t$ for $t\le t_c$ and $\varphi(t) < t$ for $t>t_c$, such that
\begin{equation}
  \label{ugen}
  \frac1n d(tn) \longrightarrow \varphi(t), \ \ \ \ t \in \R\setminus I
\end{equation}
uniformly on compact sets in probability as $n \to \infty$.
Furthermore $\varphi$ is $\cC^\infty$ everywhere except at $t=t_c$, where the acceleration satisfies $u''(t_c^+) =-\infty$. In the case of random $k$-cycles $(k\ge 2$), $I=\emptyset$ so the convergence holds uniformly on compact sets in $\R$.
\end{theorem}

\begin{rmk}\emph{We believe that $I=\emptyset$ in all cases, but our proof only guarantees this in the case of random $k$-cycles and a few other cases which we have not tried to describe precisely. Roughly speaking there is a combinatorial problem which arises when we try to estimate the distance to the identity in the case of conjugacy classes which contain several non-trivial cycles of distinct sizes (particularly when these are coprime). This is explained in more details in the course of the proof. Right now, the current result is enough to prove that there is a phase transition for $d(tn)$ when $t=t_c$, but does not prevent other phase transitions after that time.}
\end{rmk}

In the case of random $k$-cycles, we have $t_c = 1/[k(k-1)]$ and the function $\varphi$ has the following explicit expression:
\begin{equation} \label{D:phi}
\varphi(t):=
%\begin{cases}t  \text{\ \ \  if \ \ } t\le t_c \\
\displaystyle 1-\sum_{s=0}^{\infty} \frac{((k-1)s+1)^{s-2}}{s!} (kt)^s e^{-kt(s(k-1)+1)}
%  \text{ \ \ \ else}.
%\end{cases}
\end{equation}
It is a remarkable fact that for $t\le t_c$ a cancellation takes place and $\varphi(t) = t$. The case $k=2$ of random transpositions matches Theorem 4 from \cite{bd}.
In the general conjugacy class case, $\varphi$ may be described as the solution to a certain differential equation. For $t\ge 0$ and $z\in [0,1]$, let $G_t(z) = \exp(-\gg t + t\sum_{j=2}^J jk_j z^{j-1})$, and let
$\rho = \rho(t)$ be the smallest solution of the equation (in $z$): $G_t(z)=z$. Then $\varphi$ is defined by 
%\begin{theorem}\label{T:dist}
%Fix a conjugacy class $S=(k_2, \ldots, k_\ell)$, and let $$t_c:=\left(\sum_{j=2}^\ell k_j j(j-1)\right)^{-1}.$$ Fix an arbitrary $t\ge 0$. We have the convergence in probability as $n\to \infty$
%\begin{equation}
%\frac1n d(\sigma_{tn}) \longrightarrow  u(t)
%\end{equation}
%where $u$ is the function satisfying:
\begin{equation}\label{ode}
\varphi(t) = \int_0^t 1-\theta(s)^2ds.
\end{equation}
%Moreover, $u(t) = t$ for $t\le t_c$, and $u(t)<t$ for $t>t_c$. In addition, $u$ has no second derivative at $t=t_c$.
%\end{theorem}
%Recall the definition of $\cG$ in (\ref{G1}) - (\ref{G2}) as the set of times where there is a giant cycle asymptotically with high probability.
%Alternatively, $\varphi(t) = (1- u(t))/K$, where $K= \sum_{j=2}^J k_j(j-1) $ and $u(t) = \int_0^t 1-\theta(s)^2 ds$ 
It is a fact that $\theta(t)>0$ if and only if $t> t_c$, which explains why $\varphi(t)= t$ for $t \le t_c$ and $\varphi(t) < t$ for $t> t_c$.

\subsection{Heuristics}
\label{S:discussion}

The $k$-cycle random walk is a simple generalization of the random transpositions random walk on $\mathcal{S}_n$, for which the phase transition in Theorem \ref{3-cycle} was proved in \cite{bd}. Observe that any $k$-cycle $(x_1, \ldots, x_k)$ may always be written as the product of $k-1$ transpositions:
$$
(x_1, \ldots, x_k)= (x_1, x_2) \ldots (x_{k-1}, x_k)
$$
This suggests that, qualitatively speaking, the $k$-cycle random walk should behave as ``random transpositions speed up by a factor of $(k-1)$", and thus one might expect that phase transitions occur at a time that is inversely proportional to $k$. This is for instance what happens with the mixing time
\begin{equation} \label{mix}
t_{\text{mix}}=\frac1k n \log n
\end{equation}
 for the total variation distance. (This was recently proved in \cite{bsz} and was already known for $k\le 6$, the particular case $k=2$ being the celebrated Diaconis-Shahshahani theorem \cite{d-sh}); see \cite{LPW} and \cite{diaconis} for an excellent introduction to the general theory of mixing times, and \cite{lsc} in particular for mixing times of random walks on groups). It may therefore come as a surprise that $t_c = 1/[k(k-1)]$ rather than $t_c = 1/k$. As it emerges from the proof, the reason for this fact is as follows. We introduce a coupling of $(\sigma_t,t\ge 0)$ with a random hypergraph process $(H_t,t\ge 0)$ on $V=\{1, \ldots, n\}$, which is the analogue of the coupling between random transpositions and Erd\H{o}s-Renyi random graphs introduced in \cite{bd}. As we will see in more details, hypergraphs are graphs where
edges (or rather \emph{hyperedges}) may connect several vertices
at the same time. In this coupling, every time a cycle $(x_1, \ldots, x_k)$ is performed in the random walk, $H_t$ gains a hyperedge connecting $x_1, \ldots, x_k$. This is essentially the same as adding the complete graph $K_k$ on $\{x_1, \ldots, x_k\}$ in the graph $H_t$. Thus the degree of a typical vertex grows at a speed which is $k(k-1)/2 $ faster than in the standard Erd\H{o}s-Renyi random graph. This results in a giant component occurring $k(k-1)/2$ faster as well. This explains the formula $t_c^{-1} = k(k-1)$, and an easy generalisation leads to (\ref{tc}).

\medskip \textbf{Organisation of the paper:} The rest of the paper is organised as follows. We first give the proof of Theorem \ref{T:perctransp}. In the following section we introduce the coupling between $(\sigma_t,t\ge 0)$ and the random hypergraph process $(H_t,t\ge 0)$.
In cases where the conjugacy class is particularly simple (e.g. random $k$-cycles), a combinatorial treatment analogous to the classical analysis of the Erd\H{o}s-Renyi random graph is possible, leading to exact formulae. In cases where the conjugacy class is arbitrary, our method is more probabilistic in nature and the formulae take a different form ($H_t$ is then closer to the Molly and Reed model of random graphs with prescribed degree distribution, \cite{molloy-reed1} and \cite{molloy-reed2}). The proof is thus slightly different in these two cases (respectively dealt with in Section \ref{S:hypergraphs} and \ref{S:proof-conj}), even though conceptually there are no major differences between the two cases.

\section{Emergence of giant cycles in random transpositions}

In this section we give a full proof of Theorem \ref{T:perctransp}. As the reader will observe, the proof is really elementary and is based on well-known (and easy) results on random graphs. Consider the random graph process $(G_t,t\ge 0)$ on $V=\{1, \ldots, n\}$ obtained by putting an edge between $i$ and $j$ if the transposition $(i,j)$ has occurred prior to time $t$. Then every edge is independent and has probability $p_t = 1- e^{-t/{n \choose 2}}$, so $G_t$ is a realisation of the Erd\H{o}s-Renyi random graph $G(n,p_t)$.

For $t \ge 0$ and $i \in V$, let $C_i$ denote the cycle that contains $i$. Recall that if $C_i = C_j$ then a transposition $(i,j)$ yields a fragmentation of $C_i=C_j$ into two cycles, while if $C_i \neq C_j$ then the transposition $(i,j)$ yields a coagulation of $C_i$ and $C_j$. It follows from this observation that every cycle of $\sigma_t$ is a subset of one of the connected components of $G_t$. Thus let $N(t)$ be the number of cycles of $\sigma_t$ and let $\bar N(t)$ denote the number of components of $G_t$. Then we obtain
\begin{equation}
  N(t) \ge \bar N(t), \ \ \ t\ge 0.
\end{equation}
Now it is a classical and easy fact that the number $\bar N(t)$ has a phase transition at time $n/2$ (corresponding to the emergence of a giant component at this time). More precisely, let $\theta(c)$ be the asymptotic fraction of vertices in the giant component at time $cn$, so $\theta(c)$ is the survival probability of a Poisson Galton-Watson process with mean offspring $2c$ (in particular $\theta(c) = 0 $ if $c<1/2$).

Let $c>1/2$ and fix an interval of time $[t_1, t_2]$ such that $t_2 = cn$ and $t_1= t_2 - n^{3/4}$. Our goal will be to prove that a cycle of size $\delta n$ occurs during the interval $I=[t_1,t_2]$, where $\delta = \theta(c)^2/8$.

\begin{lemma} \label{L:ncRG} As $n \to \infty$,
$$
 \bar N(t_1) - \bar N(t_2) \sim (t_2 - t_1) [1- \theta^2(c)]
$$
in the sense that the ratio of these two quantities tends to 1 in probability.
\end{lemma}

\begin{proof}
 This lemma follows easily from the following observation. The total number of edges that are added during $I$ is a Poisson random variable with mean $t_2 - t_1$. Now, each time an an edge is added to $G_t$, this changes the number of components by -1 if and only if the two endpoints are in distinct components (otherwise the change is 0). Since the second largest component has size smaller than $\beta \log n$ with high probability, except on an event of probability tending to 0, throughout $[t_1, t_2]$ this occurs if and only if both endpoints are not in the giant component, which has probability uniformly close to $1-\theta^2(c)$. The law of large numbers concludes the proof.
\end{proof}

\begin{lemma}\label{L:ncSigma}
$$
\E\left(\sup_{t \le cn} |N(t) - \bar N(t)|\right) \le  3 c n^{1/2}.
$$
\end{lemma}

\begin{proof}
We already know that $N(t) \ge \bar N(t)$ for all $t \ge 0$. It thus suffices to control that the excess number of cycles is never more than $4 n^{1/2}$ in expectation. Note first that there can never be more than $n^{1/2}$ cycles of size greater than $n^{1/2}$. Thus it suffices to count the number $N^{ex}_{\downarrow}(t)$ of excess cycles of size $\le n^{1/2}$:
$$
|N(t) - \bar N(t)| \le N^{ex}_{\downarrow}(t) + n^{1/2}.
$$
 These excess cycles of size $\le n^{1/2}$ at time $t$ must have been generated by a fragmentation at some time $s\le t$ where one of the two pieces was smaller than $n^{1/2}$. But at each step, the probability of making such a fragmentation is smaller than $2n^{-1/2}$. Indeed, given the position of the first marker $i$, there are at most $2n^{1/2}$ possible choices for $j$ which result in a fragmentation of size smaller than $n^{1/2}$. To see this, note that if a transposition $(i,j)$ is applied to a permutation $\sigma$, and $C_i = C_j$, so $\sigma^k(i) =j$, then the two pieces are precisely given by $(\sigma^0(i), \ldots, \sigma^{k-1}(i))$ and $(\sigma^0(j), \ldots, \sigma^{|C|-k-1}(j))$. Thus to obtain a piece of size $k$ there are at most two possible choices, which are $\sigma^k(i)$ and $\sigma^{-k}(i)$. Thus $\E(F_\downarrow(cn)) \le cn \cdot 2n^{-1/2}$, where $F_\downarrow(cn)$ is the total number of fragmentation events where one of the pieces is smaller than $n^{1/2}$ by time $cn$. Since
$$
\sup_{t\le cn} N^{ex}_{\downarrow}(t) \le F_{\downarrow}(cn)
$$
this finishes the proof.
\end{proof}

\begin{proof}[Proof of Theorem \ref{T:perc2}.] Appying Markov's inequality in Lemma \ref{L:ncSigma}, we see that since $n^{1/2} \ll n^{3/4} = t_2 - t_1$, we also have
$$
N(t_1) - N(t_2) \sim (t_2 - t_1) (1- \theta^2(c))
$$
in probability, by Lemma \ref{L:ncRG}. On the other hand, $N(t)$ changes by -1 in the case of a coalescence and by $+1$ in the case of a fragmentations. Hence $N(t_1) - N(t_2) = \text{Poisson}(t_2 - t_1) - 2 F(I)$, where $F(I)$ is the total number of fragmentations during the interval $I$. We therefore obtain by the law of large numbers for Poisson random variables:
$$
F(I) \sim \frac12 (t_2 - t_1)\theta(c)^2.
$$
But observe that to if $F(I)$ is large, it cannot be the case that all cycles are small - otherwise we would very rarely pick $i$ and $j$ in the same cycle. Hence consider the event $E=\{\tau_\delta <cn\}$. On $E^\complement$, the maximal cycle size throughout $I$ is no more than $\delta n $. Hence at each transposition, the probability of making a fragmentation is no more than $\delta$. By the law of large numbers, on the event $E^\complement$, it must be that $F(I) \le 2 \delta (t_2 - t_1)$. Since $2 \delta = \theta(c)^2/4$, it follows immediately that $\P(E^\complement) \to 0$ as $n\to \infty$. This completes the proof.
\end{proof}

\begin{rmk}
This proof is partly inspired by the calculations in Lemma 8 of \emph{\cite{geometry}}. 
\end{rmk}

\section{Random hypergraphs and Theorem \ref{3-cycle}.}
\label{S:hypergraphs}

We now start the proof of Theorem \ref{3-cycle}. We first review some relevant definitions and results from random hypergraphs.

A \emph{hypergraph} is a graph where edges can connect several
vertices at the same time. Formally:

\begin{defi} A hypergraph $H=(V,E)$ is given by a set $V$ of vertices and a subset
$E$ of $\mathcal{P}(V)$, where $\mathcal{P}(V)$ denotes the set of
all subsets of $V$. The elements of $E$ are called hyperedges. A
$d$-regular hypergraph is a hypergraph where all edges connect $d$
vertices, i.e. for all $e\in E$, $|e|=d$.
\end{defi}

\medskip For a given $d\ge 2$ and $0<p<1$, we call $\mathbf{G}_d(n,p)$ the probability distribution
on $d$-regular hypergraphs on $V=\{1,\ldots,n\}$ where each
hyperedge on $d$ vertices is present independently of the other
hyperedges with probability $p$. Observe that when $d=2$ this is
just the usual Erd\H{o}s-Renyi random graph case, since a
hyperedge connecting two vertices is nothing else than a usual
edge. For basic facts on Erd\H{o}s-Renyi random graphs, see e.g.
\cite{bb-book}.

The notion of a hypertree needs to be carefully formulated in what follows. We start with the $d$-regular case. The excess $ex(H)$ of a given $d$-regular hypergraph $H$ is defined to be
\begin{equation}\label{exdef}
ex(H)=(d-1)h-r
\end{equation}
where $r=|H|$ and $h$ is the number of edges in $H$.

Observe that if $H$ is
connected then $ex(H)\ge -1$.
\begin{defi} \label{ex}  We call a connected $d$-regular hypergraph $H$ a \emph{hypertree} if $ex(H)=-1$.
\end{defi}

Likewise if $ex(H)=0$ and $H$ is connected we will say that
$H$ is \emph{unicyclic} and if the excess is positive we will say
that the component is \emph{complex}.

\begin{rmk} This is the definition used by Karo\'nski
and Luczak in \cite{kl02}, but differs from the definition in their older paper \cite{kl93}
where a hypertree is a connected hypergraph such that removing any hyperedge would make it disconnected.
\end{rmk}

In the case where $H$ is not necessarily regular, the excess of a connected hypergraph $H$ made up of the hyperedges $h_1, \ldots, h_n$ is defined to be $ex(H) =  \sum_{i=1}^n (|h_i| - 1) - |H| $, where $|h_i|$ denotes the size of the hyperedge $h_i$ and $|H$ is the cardinality of the vertex set of $H$. Then $ex(H) \ge 1$ and $H$ is said to be a hypertree if $ex(H) = -1$.

\subsection{Critical point for random hypergraphs}

We start by recalling a theorem by Karo\'nski and Luczak \cite{kl02} concerning the emergence of a giant connected component in a random hypergraph process $(H_t,t\ge 0)$ where random hyperedges of degree $d \ge 2$ are added at rate 1.

\begin{theorem} \label{KL} Let $c>0$ and let $t=cn$.
\begin{itemize} \item[-] When $c<c_d = 1/[d(d-1)]$ then $a.a.s$ then $H_t$ contains only trees
and unicyclic components. The largest component has size $O(\log n)$ with high probability.

\item[-] When $c>c_d$ then there is a.a.s a unique complex
component, of size $\theta n$ asymptotically, where $\theta = \theta_d(c)>0$. All other
component are not larger than $O(\log n)$ with high probability.
\end{itemize}
\end{theorem}

Note that if $c<c_d$ the number of unicyclic components is no more than $C' \log n$ for some $C'>0$ which depends on $c$. Indeed, at each step the probability of creating a cycle is bounded above by $C  \log n /n$ since the largest component is no more than $O(\log n)$ prior to time $cn$. Since there are $O(n)$ steps this proves the claim. We will need a result about the evolution of the number of components $\bar N(t)$ in $(H_t,t\ge 0)$.

\begin{prop} \label{clusters hypergraph}
Let $t>0$. Then as $n\to \infty$,
$$
\frac1n \bar N (tn) \longrightarrow_p
  \sum_{h=0}^{\infty} \frac{((d-1)h+1)^{h-2}}{h!} (dt)^h e^{-dt(h(d-1)+1)}
$$
\end{prop}

\begin{proof}
Note first that, by monotonicity of the number of clusters and continuity of the function in the right-hand side, it suffices to establish this result when $t \neq 1/[d(d-1)]$.
Moreover, by Theorem \ref{KL} and since there are no more than $C \log n$ unicyclic components it is enough to count the number of hypertrees $\tilde N(s)$ smaller than $C \log n$ in
$H_s$ where $s=tn$.
We will first compute the expected value and
then prove a law of large numbers using a second moment method.

\mn Let $h\ge 0$, we first compute the number of hypertrees with
$h$ hyperedges ($h=0$ corresponds to isolated vertices). These
have $r=(d-1)h+1$ vertices. By Lemma 1 in Karo\'nski-Luczak \cite{kl97}, there
are
\begin{equation}\label{hyper count}
\frac{(r-1)! r^{h-1}}{h![(d-1)!]^h}
\end{equation}
trees on $r=(d-1)h+1$ labeled vertices (this is the analogue to Cayley's
(1889) well-known formula that there are $k^{k-2}$ ways to draw a
tree on $k$ labeled vertices). If $T$ is a given hypertree with $h$ edges labelled by elements of $V= \{1, \ldots, n\}$,
there are a
certain number of conditions that must be fulfilled in order for
$T$ to be one of the components of $G$: (i) The $h$ hyperedges of $T$ must be
open, (ii) ${r \choose d} -s$ hyperedges must be closed
inside the rest of $T$, (iii) $T$ must be disconnected from the rest
of the graph, which requires closing $r{ n-r \choose d-1}$
hyperedges.

\medskip Now, remark that at time $s=tn$, because the individual
Poisson clocks are independent, each hyperedge is present
independently of the others with probability
%\begin{equation}\label{open}
$p=1- \exp\left(-s/{{n \choose d}}\right) \sim d!t/n^{d-1}.$
%\end{equation}
It follows that the probability that $T$ is one of the components of $H_t$ is
\begin{equation}\label{hyper count 2}
p^h(1-p)^{{r \choose d}-h+r{n-r \choose d-1}}.
\end{equation}
Hence the expected number of trees in $H_s$ with $h$
edges is
\begin{align}
\E[\tilde N_h(tn)] &= {n \choose r}\frac{(r-1)!r^{h-1}}{h![(d-1)!]^h} p^h(1-p)^{{r \choose
d}-h+r{{n-r \choose d-1}}} \label{EN_h}\\
%(1-p)^{{r \choose d}-s+r{n-r \choose d-1}} & \sim \left(1- \frac{d!t}{n^{d-1}}\right)^{r {n-1 \choose d-1}}\\
&\sim n \frac{r^{h-2}}{h!} (d t)^h e^{-dr t} \nonumber
\end{align}
%By Fatou's Lemma, summing over all possible values of $h\ge 0$ we get (still denoting $r=(d-1)h +1$ in this sum):
%$$
%\liminf_{n\to\infty}\frac1n \E(\tilde N(tn))\ge
%\sum_{h=0}^{\infty} \frac{r^{h-2}}{h!} (dt)^h e^{-d r t}.
%$$
Write $\cC$ for the set of connected components of $H_t$. Note that if $T_1$ and $T_2$ are two given hypertrees on $V$ with distinct vertex sets and with $h$ hyperedges each, then
$$
\P(T_1 \in \cC \text{ and } T_2 \in \cC) = \frac{\P(T \in \cC)^2}{(1-p)^{r^2}}.
$$
From this we deduce that $\cov(\indic{T_1 \in \cC}, \indic{T_2 \in \cC}) \to 0$ and that $\var(\tilde N_h(s)) = o( n^2)$. Thus, by Chebyshev's inequality:
\begin{equation}\label{small}
\frac1n \sum_{h=0}^{h_0} \tilde N_h(s) \longrightarrow_p  \sum_{h=0}^{h_0} \frac{((d-1)h+1)^{h-2}}{h!} (dt)^h e^{-dt(h(d-1)+1)},
\end{equation}
in probability as $n \to \infty$. The end of the proof of the proposition now follows from (\ref{small}) and the following bound:
\begin{equation}
\limsup_{h_0 \to \infty} \limsup_{n\to \infty} \frac1n  \E(\tilde N_{> h_0}(s)) = 0.
%{n \choose r}\frac{(r-1)!r^{h-1}}{h![(d-1)!]^s} p^s(1-p)^{{r \choose d}-h+r{n-r
%\choose d-1}} = 0.
\label{large}
\end{equation}
where $\tilde N_{> h_0}(s) = \sum_{h = h_0+1}^{C \log n} \tilde N_h(s)$.
Indeed, for every $\eps>0$ and $\eta>0$, we can choose $h_0$ large enough such that the finite sum in the right-hand side of \eqref{small} lies within $\eps$ of the infinite series. We then choose $n$ large enough so that
$
\E(\tilde N_{> h_0}(s))/n \le \eps \eta,
$
whence by Markov's inequality:
$$
\P( \frac1n \tilde N_{>h_0}(s) > \eps n) \le \eta
$$
We now conclude using (\ref{small}). To obtain the bound \eqref{large} we use \eqref{EN_h}, from which it follows (using ${n \choose r} \le n^r /r!$ and $1-e^{-x} \le x$),
$$
\E(\tilde N_h(s)) \le n \frac{(dt)^h r^{h-2}}{h!} \exp\left(-  s r { n - r \choose d-1}/{n \choose d}\right).
$$
But since $r  = (d-1)h +1 \le C \log n$, we see that
$$
\E(\tilde N_h(s)) \le n \frac{(dt)^h r^{h-2}}{h!} \exp( -rd t + o(1))
$$
where the term $o(1)$ is uniform in $r \le C \log n$. Using Stirling's formula we obtain a uniform exponential bound for $\E(\tilde N_h(s)/n)$ provided that $t \neq 1/[d(d-1)]$. \eqref{large} now follows.
\end{proof}

\subsection{Bounds for the Cayley distance on the symmetric group}

\medskip In the case of random transpositions we had the convenient
formula that if $\sigma \in \mathcal{S}_n$ then
$d(o,\sigma)=n-\#\text{cycles}$, a formula originally
due to Cayley. In the case of random $k$-cycles with $k\ge 3$, unfortunately there is to our knowledge no exact
formula to work with. However this formula stays approximately true, as shown by the following proposition.

\begin{prop} \label{bounds distance}
Let $k\ge 3$ and let $\sigma \in \cS_n$. (If $k$ is odd, assume further that $\sigma \in \cA_n$). Then
$$
 \frac1{k-1}(n-|\sigma|) \le d(o,\sigma) \le \frac1{k-1}(n-|\sigma|) + C(k) |R_k(\sigma)|
$$
where $|\sigma|$ is the number of cycles of $\sigma$, $C(k)$ is a universal constant depending only on $k$, and $R_k(\sigma)$ is the set of cycles of $\sigma$ whose length $\ell \neq 1$ mod $k-1$.
\end{prop}

\begin{proof} For simplicity we consider only the case $k=3$. Thus let $\sigma \in \cA_n$.
For each cycle of odd length $(i_1,\ldots,i_{2r+1})$
we can write
$$
(i_1,\ldots,i_{2r+1}) = (i_1, i_2,
i_3)(i_3,i_4,i_5)\ldots(i_{2r-1},i_{2r},i_{2r+1})
$$
which has exactly $r$ 3-cycles factors. Now, because $\sigma \in
\mathcal{A}_n$, the number of cycles of even length must be even.
So let $(i_1,\ldots,i_{2r})(j_1,\ldots,j_{2m})$ be a pair of even
cycles. Then we start by building
$$(i_1,i_2)(j_1,j_2)=(i_1,i_2,j_1)(i_2,j_1,j_2)$$
in two moves and then completing each of the cycle in the same way
as above. The total number of moves to build this pair of cycles
is thus $2+(r-1)+(m-1) = r+m$. It follows that $\sigma$ can be made up of at most
$$
\sum_{c \notin R_2(\sigma)} \half(|c|-1) + \sum_{c \in R_2(\sigma)} \half|c|= \half(n -|\sigma|)+\half|R_2(\sigma)|.
$$
This gives the upper-bound. On the other hand, multiplying $\sigma$ by a 3-cycle can
create at most two new cycles. Hence, after $p$ multiplications
the resulting permutation cannot have more than $|\sigma| + 2p$
cycles. Therefore the distance must be at least that $k_0$ for
which $|\sigma| + 2k_0\ge n$, since the identity permutation has exactly $n$ cycles. The lower-bound follows. \end{proof}

\subsection{Phase transition for the $3$-cycle random walk}

We now finish the proof of Theorem \ref{3-cycle} in the case of random $k$-cycles.

\begin{proof}[Proof of Theorem \ref{3-cycle} if $k_j = \delta_{{k,j}}$]

The proof follows the lines of Lemma \ref{L:ncSigma}. Let $N(t)$ be the number of cycles of $\sigma$ and let
$\bar N(t)$ be the number of components in $H_t$, where $(H_t,t\ge 0)$ is the random $k$-regular hypergraph process obtained by adding
the edge $\{x_1, \ldots, x_k\}$ whenever the $k$-cycle $(x_1, \ldots, x_k)$ is performed. Then note again that every cycle of $\sigma_t$ is a subset of a connected component of $H_t$, so $N(t) \ge \bar N(t)$. (Indeed, this property is a deterministic statement for transpositions, and a sequence of random $k$-cycles can be decomposed as a sequence $(k-1)$ times as long of transpositions.).

%%\medskip Note first that, on the event that no $k$-cycle is chosen exactly twice by time $t$ (an event of high probability if $t=cn$ for some $c>0$), any hypertree $T$ in $H_t$ corresponds exactly to a cycle of $\sigma_t$. Indeed, let $e_1, \ldots, e_p$ be the hyperedges of $T$, in the chronological order with which they were added. Then corresponding $k$-cycles  $\gamma_1, \ldots, \gamma_p$ may only have coagulated cycles of $\sigma$, since if there ever was a fragmentation event, say during the multiplication by $\gamma_i$, then the hyperedge $e_i$ would be connecting two points already in the same cycle, and hence we would get that $ex(T) \ge 0$, which is impossible.

Repeating the argument in Lemma \ref{L:ncSigma}, we see that
\begin{equation}\label{cl-cy}
n^{-3/4}\left(\sup_{t \le cn} |N(t) - \bar N(t) |\right) \to 0,
\end{equation}
in probability.
%To see where this comes from, recall that any $k$-cycle $\gamma$ can be decomposed as a product of transpositions
%$$
%\gamma = (x_1, \ldots, x_k) = (x_1, x_2) \ldots (x_{k-1}, x_k).
%$$
%We can thus view $\sigma_t$ as the product $\prod_{i=1}^m \tau_i$ of a number of transpositions, where $m$ is a certain random variable and the transpositions are not independent but if $1\le j \le m-1$ which is not a multiple of $k$, and if $i$ is the largest multiple of $k$ such that $i \le j$, then conditionally on $m$ and $\tau_i = (x_i, x_{i+1}), \ldots, \tau_{j-1} = (x_{j-1}, x_j)$, then $\tau_{j} = (x_j,y)$ where $y$ is uniform on $V \setminus\{x_i, \ldots, x_j\}$. Thus for any $1\le j \le m$, the conditional probability, given $\tau_1, \ldots, \tau_{j-1}$, that $\tau_j$ yields a fragmentation with one of the two pieces being smaller than $\sqrt{n}$ is no more than $4\sqrt{n}/(n-k) \le 5 n^{-1/2}$, since at most $4\sqrt{n}$ choices of $y$ will have this effect. Since $m$ exceeds $Cn$ with exponentially low probability for some $C>0$, the statement \eqref{cl-cy} follows. Using Proposition \ref{bounds distance}, Proposition \ref{clusters hypergraph} and the lower-bound (\ref{cl-cy}) gives us directly the lower-bound for Theorem \ref{3-cycle} in the case of $k$-cycles. The upper-bound follows by noting that
This is proved in greater generality (i.e., for arbitrary conjugacy classes) in Lemma \ref{L:clusters2}. Moreover, for any $c \in R_k(\sigma_t)$ must have been generated by fragmentation at some point (otherwise the length of cycles only increases by $k-1$ each time). Thus $R_k(\sigma_t) \le N(t) - \bar N(t)$, and Theorem \ref{3-cycle} now follows.

\end{proof}

\section{Proofs for general conjugacy classes}
\label{S:proof-conj}

\subsection{Random graph estimates}

Let $\g=(k_2, \ldots, k_J)$ be our fixed conjugacy class.
A first step in the proof of Theorems \ref{3-cycle} and \ref{T:perc2} in this general case is again to associate a certain random graph model to the random walk. As usual, we put a hyperedge connecting $x_1, \ldots, x_k$ every time a cycle $(x_1\ldots x_k)$ is performed as part of a step of the random walk. Let $H_s$ be the random graph on $n$ vertices that is obtained at time $s$. A first step will to prove properties of this random graph $H_s$ when $s=tn$ for some constant $t>0$. Recall our definition of $t_c$:
\begin{equation}\label{tc2}
t_c^{-1}=\sum_{j=2}^J k_j j(j-1),
\end{equation}
and that $1-\theta$ be the smallest solution of the equation (in $z$): $G_t(z)=z$, where
\begin{equation}\label{G_t}
G_t(z) = \exp(-t\sum_{j=2}^J j k_j + t\sum_{j=1}^J jk_j z^{j-1}).
\end{equation}

\begin{lemma}\label{L:giant}
If $t<t_c$ then there exists $\beta>0$ such that all clusters of $H_{tn}$ are smaller than $\beta \log n$ with high probability. If $t> t_c$, then there exists $\beta>0$ such that all but one clusters are smaller than $\beta \log n$ and the largest cluster $L_n(t)$ satisfies
$$\frac{L_n(t)}n \overset{n\to \infty}\longrightarrow\theta(t)$$
in probability.
\end{lemma}

\begin{proof}
We first consider a particular vertex, say $v \in V$, and ask what is its degree distribution in $H_{tn}$.
%For a particular $2 \le j \le J$, we note that after $m$ steps of the walk, the number $E^n_j$ of hyperedges of size $j$ that are connected to $v$ it is Binomial$(sk_j, j/n +o(1/n))$. Moreover, at time $tn$, since $m$ is a Poisson random variable with mean $tn$, it follows that $E^n_j$ converges to a Poisson random variable with mean $tjk_j$. In fact, the joint processes $(E^n_j(sn),s\ge 0)_{2\le j \le J}$ converge to independent towards Poisson processes $(E^\infty_j(s), s \ge 0)_{2\le j \le J}$ with respective rates $jk_j$. This follows, e.g., from Theorem 2.4 in Chapter VII of Jacod in Shiryaev \cite{J-S}.
Write $\sigma_t = \gamma_1 \ldots \gamma_{N_t}$ where $(\gamma_i,i\ge 1)$ is a sequence of i.i.d. permutations uniformly distributed on $\g$, and $(N_t,t\ge 0)$ is an independent Poisson process. Note that for $t\ge 0$, $\#\{n\le N_t: v \in \text{Supp}(\gamma_i)\}$ is a Poisson random variable with mean $t \sum_{j=2}^J j k_j /n$. Thus by time $tn$, the number of times $v$ has been touched by one of the $\gamma_i$ is a Poisson random variable with mean $t \sum_{j=2}^J j k _j$. For each such $\gamma_i$, the probability that $v$ was involved in a cycle of size exactly $\ell$ is precisely $\ell k_\ell / \sum_{j=2}^J jk_j$. Thus, the number of hyperedges of size $j$ that contain $v$ in $H_{tn}$ is $P_j$, where $(P_j,j=2,\ldots, J)$ are independent Poisson random variables with parameter $tjk_j$. Since each hyperedge of size $j$ corresponds to $j-1$ vertices, we see that the degree of $v$ in in $H_{tn}$, $D_v$, has a distribution given by
\begin{equation}\label{progeny}
D_v=\sum_{j=2}^\ell (j-1) P_j.
\end{equation}
Now, note that by definition of $t_c$ (see \eqref{tc2}),
$$
\E(D_v)> 1 \iff t> t_c.
$$
%Note also that, conditionally on the degree count $d_i= \sum_{v \in V} \indic{D_v = i}$, for $0\le i \le n$, $H_{tn}$ has the same connectivity properties as a uniform random graph with this degree count. We may thus apply the main result of Molloy and Reed (\cite{molloy-reed1}, also recalled in Theorem 3 of \cite{molloy-reed2}). To apply this result, we must verify that the degree count is sparse and \emph{well-behaved} (in the terminology of \cite{molloy-reed1}, or see Definitions 3-6 of \cite{molloy-reed2}), and that $d_i(n) = 0$ for $i\ge n^{1/4 - \eps}$, with high probability, for some $\eps>0$. The sparsity and well-behavedness of the degree count is easy to check from (\ref{progeny}). The last condition follows from easy large deviations for Poisson random variable, since $\P(D_v \ge n^{1/5}) \le \exp(- cn^{1/5})$ for some $c>0$. Thus $\P(d_i(n) \neq 0 \text{ for some } i \ge n^{1/5}) \le n \exp(- cn^{1/5}) \to 0. $ Molloy and Reed's main result in \cite{molloy-reed1} (see also Theorem 3 in \cite{molloy-reed2}) guarantees that there is only one component greater than $\beta \log n$ with high probability for some $\beta>0$, and that this component has size at least $\theta_1 n$, where $\theta_1 >0$, if the above assumptions are satisfied, and if $Q>0$, where $Q = \sum_{i=2}^\infty i(i-2) \lambda_i$ and $\lambda_i = \lim_{n\to \infty} d_i(n) / n$. Thus $i\lambda_i = \P(D_v = i)$, and one easily check that $Q>0$ if and only if $\E(D_v) >1$.
The proof of Theorem 3.2.2 in Durrett \cite{durrett-rg} may be adapted almost \emph{verbatim} to show that there is a giant component if and only if $\E(D_v)>1$, and that the fraction of vertices in the giant component is the survival probability of the associated branching process. Note that the generating function associated with the progeny (\ref{progeny}) is
\begin{align*}
G_t(z)&:=\E(z^D) = \prod_{j=2}^J\E(z^{(j-1)P_j})
=\prod_{j=2}^\ell \exp(tjk_j(z^{j-1}-1)) \\
& =\exp\left(-\gg t +t \sum_{j=2}^J jk_j z^{j-1}\right)
\end{align*}
thus $\rho(t)=1-\theta(t)$ is the smallest root of the equation $G_t(z)=z$. From the same result one also gets that the second largest cluster is of size no more than $\beta \log n$ with high probability, for some $\beta>0$.
\end{proof}

Let $\nc(s)$ be the number of clusters at time $s$ in $H_s$, and let $u_n(t)= \frac1n\E(\nc(tn))$. Define a function $u(t)$ by putting:
\begin{equation}
\label{Def:u}
  u(t) = 1 - K \int_0^t 1- \theta(s)^2 ds,
\end{equation}
where $K:= \sum_{j=2}^J k_j(j-1)$, and note that that $u(0) = 1$, for $t < t_c$ we have $u(t) = 1- Kt$, and  $u(t) > 1- Kt $ for $t > t_c$.

\begin{lemma}\label{L:clusters}
As $n\to \infty$, we have
$$
u_n(t) \longrightarrow u(t),
$$
uniformly on compacts in probability.
\end{lemma}

\begin{proof}
Let $H$ denote a hypergraph on $\{1,\ldots, n\}$, and let $h=h_1\cup \ldots\cup h_\ell$ be a set of hyperedges. Denote by $H'=H+h$ the graph obtained from $H$ by adding the hyperedges $h_1, \ldots, h_\ell$ to $H$.
Let $(x_1, \ldots, x_n)$ be a discrete partition of unity, i.e., a non-increasing sequence of numbers such that $\sum_{i=1}^n x_i =1$ and $ nx_i$ is a nonnegative integer. Define a function $f(x_1, \ldots , x_n)$ as follows. Let $H$ be any hypergraph for which $x_i$ are the normalized cluster sizes. Let $h=h_1 \cup \ldots \cup h_\ell$ be a collection of hyperedges of sizes $2,3, \ldots, J$ (with size $j$ being of multiplicity $k_j$), where the hyperedges $h_i$ are sampled uniformly at random without replacement from $\{1, \ldots, n\}$. Let $H'=H+h$. Then we define $f$ by putting
$$
f(x_1,\ldots,x_n):= \E( |H'| - |H|)
$$
where $|H|$ denotes the number of clusters of $H$. Then we have that
\begin{equation}
M_t:= \frac1n|H(tn)|- \int_0^t f(x_1(sn),\ldots, x_n(sn)) ds
\end{equation}
is a martingale, if $(x_1(s), \ldots, x_n(sn))$ denote the ordered normalized cluster sizes of $H(s)$. (Note that $M_0=1$.) Thus, taking expectations,
$$
u_n(t)=1+ \int_0^t\E[ f(x_1(sn),\ldots, x_n(sn))]ds
$$
We claim that, as $n\to \infty$, for every $s$ fixed,
\begin{equation}\label{gain}
\E(f(x_1(sn), \ldots, x_n(sn)) \to - K(1-\theta(s)^2).
\end{equation}
where $K = \sum_{j=2}^J k_j(j-1)$. To see this, note that for every hyperedge $h=\{i_1, \ldots, i_j\}$ of size $2 \le j \le J$ which is added to the graph, the increase in the number of clusters is the same as if we successively add the edges $\{i_1,i_2\}, \ldots \{i_{j-1}, i_j\}$. Let us compute the expected gain when adding the edge $\{i_{k-1},i_k\}$. Summing over $k$ gives us the expected gain after adding the edge $h$ by linearity of the expectation, and summing over hyperedges will give us the value of $f$. Now, condition on what happens by the time we add the edge $\{i_{k-1},i_k\}$. If the cluster sizes are $(x_1, \ldots)$, then, either $i_k$ falls into the same component as $i_{k-1}$, in which case the number of components does not change, or $i_k$ falls in a different component, in which case, the number of clusters decreases by 1. Hence, the expected gain at this stage is
$$
\sum_{i\ge 1} x_i[-(1-x_i)] = -1 + \sum_{i\ge 1} x_i^2.
$$
As $n \to \infty$, by Lemma \ref{L:giant}, this converges to $-1 + \theta(s)^2$. (Note in particular that this limit is independent from what happened during the earlier edges added to $H$). Since there are $(j-1)$ edges to add for a hyperedge of size $j$ and $k_j$ such hyperedges, (\ref{gain}) follows. Using the Lebesgue convergence theorem, we deduce that,
$$
u_n(t) \to 1 - K\int_0^t (1-\theta^2(s))ds = u(t).
$$
To obtain convergence in the u.c.p. sense (uniform on compacts in probability), we note that
\begin{equation}\label{variance}
\var(|H'|-|H|) \le C
\end{equation}
for some constant $C$ which depends only on $(k_2, \ldots, k_\ell)$, since $|H'|$ may differ from $|H|$ only by a bounded amount. Now, by Doob's inequality, if $\bar M_s = n(M_s-1)$:
\begin{align}
\P\left(\sup_{s\le t} |(M_s-1)| > \eps\right) & = \P\left( \sup_{s\le t} |\bar M_s|^2 >n^2 \eps^2\right) \nonumber \\
& \le \frac{4\var(\bar M_t)}{n^2\eps^2} \nonumber \\
& \le \frac{4C}{n\eps^2}.\label{doob}
\end{align}
The last line inequality is obtained by conditioning on the number of steps $N$ between times 0 and $tn$, noting that  after each step, the variance of $\bar M_t$ increases by at most $C$ by (\ref{variance}). Hence, to conclude the proof of Lemma \ref{L:clusters}, it suffices to show that we have the convergence:
\begin{equation}
\label{gain2}
\int_0^t f(x_1(sn),\ldots) ds \longrightarrow  -K \int_0^t (1-\theta(s)^2) ds,\ \ u.c.p.
\end{equation}
\end{proof}
This is a direct consequence of the fact that as $n\to \infty$:
$$
f(x_1(sn),\ldots) \longrightarrow  -K(1-\theta(s)^2) , \ \ u.c.p.
$$
which itself follows from pointwise convergence in probability, monotonicity in $s$, and the fact that the limiting function is continuous. (Monotonicity comes from a simple coupling argument, using the fact that $H(t)$ is a purely coalescing process).

\subsection{Random walk estimates}

\begin{lemma} \label{L:clusters2} Let $N(t)$ be the number of cycles of $\sigma(tn)$. Then we have, as $n\to \infty$:
\begin{equation}\label{clusters2}
\frac1{n^{3/4}} (N(tn) - \bar N(tn)) \longrightarrow 0, \ \ u.c.p.
\end{equation}
\end{lemma}

\begin{proof}
This is very similar to Lemma \ref{L:ncSigma}. Say that a cycle is large or small, depending on whether it is bigger or smaller than $\sqrt{n}$. To start with, observe that there can never be more than $\sqrt{n}$ large cycles. As usual, we have that $N(t) \ge \bar N(t)$, and we let $N^{ex}(t) = N(t) - \bar N(t)$ be the excess number of cycles. This in turn can be decomposed as $N^{ex}(t) = N^{ex}_{\uparrow}(t) + N^{ex}_{\downarrow}(t)$, where the subscripts $\uparrow$ and $\downarrow$ refer to the fact that the cycles are either small or large. Thus we have
$$
N^{ex}_{\uparrow}(t) \le \sqrt{n},
$$
and the problem is to control $N^{ex}_\downarrow(t)$. Writing every cycle of size $j$ as a product of $j-1$ transpositions, we may thus write $\sigma_t = \prod_{i=1}^{m_t} \tau_i$, for a sequence of transpositions having a certain distribution (they are not independent). Then $N^{ex}_\downarrow(t) \le F_\downarrow(t)$, where $F_\downarrow(t)$ is the number of times $1  \le i \le m$ that the transpositions $\tau_i$ yields a fragmentation event for which one of the fragments is small. However, conditionally on $\tau_1, \ldots, \tau_{i-1}$, the conditional probability that $\tau_i$ yields such a fragmentation is still bounded by $4n^{-1/2}$. Since $m_t = K N_t$, where $K = \sum_{j=2}^J (j-1) k_j \ge 1$ and $N_t$ is a Poisson random variable with mean $t$, it follows that
$$
\E( \sup_{s\le tn} F_\downarrow(s) ) \le 4Kt\sqrt{n}
$$
Thus by Markov's inequality,
\begin{equation}\label{frag}
\P\left(\sup_{s\le tn} F_{\downarrow}(s) > n^{3/4}\right) \longrightarrow 0.
\end{equation}
Hence,
$
n^{-3/4}| N(tn) - \bar N(tn) |
$
converges to 0 u.c.p, which concludes the proof by Lemma \ref{L:clusters2}.
\end{proof}

Note in particular that by combining Lemma \ref{L:clusters} with Lemma \ref{L:clusters2}, we get that
\begin{equation}\label{clusters3}
\frac1n N(tn) \to u(t),\ \ \ u.c.p.
\end{equation}

\begin{lemma}\label{L:large}
%Let $Y_n(s)$ be the size of the largest cluster at time $s$, and let
%$$
%\Lambda_n(t) := \sup\left\{\frac{Y(sn)}n, 0 \le s \le t \right\}.
%$$
%Then with high probability, for all $t>t_c$,
Let $t>t_c$. Then $\tau_\delta < tn$ with high probability, where
\begin{equation}\label{largest}
\delta:=\frac{2^K}t \int_0^t \theta^2(s)ds >0,
\end{equation}
where $K = \sum_{j=2}^J (j-1) k_j$.
\end{lemma}

\begin{rmk}
Note that Lemma \ref{L:large} immediately implies Theorem \ref{T:perc2}.
\end{rmk}

\begin{proof}
The idea is to say that, since we know that the number of cycles is approximately the number of clusters in the random graphs, this implies a nonlinearity in the behaviour of this number. In turns, this means there are many fragmentations and thus that there are some large clusters.

To formalize this, assume that a permutation $\sigma$ has a cycle structure $(C_1, \ldots, C_r)$ and that $x_1, \ldots, x_r$ are the normalized cycle sizes, i.e., $x_i = |C_i| /n$. Define a function $g(x_1, \ldots, x_r)$ by putting
$$
g(x_1, \ldots, x_r):= \E( |\sigma'| - |\sigma|),
$$
where $\sigma' = \sigma \cdot \gamma$ and $\gamma$ is a uniform random element from $\Gamma$, while $|\sigma|$ denotes the number of cycles of $\sigma$. Then if we define a process
$$
M'_t = \frac1n N(tn) - \int_0^t g(x_1(sn), \ldots) ds,
$$
then $(M'_t, t \ge 0)$ is a martingale started from $M'_0=1$. Moreover, writing $\tau= \tau_1 \cdot \ldots \cdot \tau_K$, where $\tau_i$ are transpositions, and if we let $\sigma_i = \sigma \cdot \tau_1 \ldots \tau_i$, so that $\sigma_0 = \sigma$ and $\sigma_K= \sigma'$, then
$$
g(x_1, \ldots, x_r) = \sum_{i=1}^K \E(|\sigma_i| -|\sigma_{i-1}|).
$$
Recall that the transposition $\tau_i$ can only cause a coalescence or a fragmentation, in which case the number of cycles decreases or increases by 1. If the relative cycle sizes of $\sigma_{i-1}$ are given by $(y_1, \ldots, y_r)$, it follows that
$$
-1 \le \E(|\sigma_i| -|\sigma_{i-1}|) \le -1 + 2 y^*_i \frac{n}{n-i+1},
$$
where $y^*_i = \max (y_1, \ldots, y_r)$. Moreover, $y^*_i \le 2^{i} y^*_0$.

From this we obtain directly that with high probability (uniformly on compact sets)
\begin{equation}\label{lb}
\int_0^t g(x_1(sn), \ldots) ds \le \int_0^t K\left[-1 + 2^K  x^*(sn) \right]ds,
\end{equation}
where $x^*(s) = \max(x_1(s), \ldots, x_r(s))$. On the other hand, using Doob's inequality in the same way as (\ref{doob}), we also have:
\begin{equation}\label{doob2}
\P\left(\sup_{s\le t} |(M'_s-1)| > \eps\right)
 \le \frac{4C}{n\eps^2}.
\end{equation}
Combining this information with (\ref{clusters3}), we obtain, with high probability uniformly on compact sets:
\begin{equation}
\int_0^t\left[-1 + 2^K x^*(sn) \right]ds \ge \int_0^t -1+\theta^2(s)ds.
\end{equation}
From this we get, since $\sum_{i=1}^r x_i(sn)^2 \le  \Lambda_n(t)$, with high probability
\begin{equation}
t2^K \sup_{s\le tn}x^*(s) \ge \int_0^t \theta^2(s)ds,
\end{equation}
i.e., $\tau_\delta \le tn$.
\end{proof}

\subsection{Distance estimates}

We are now ready to prove that
$$
d(\sigma_{tn}) \longrightarrow \varphi(t),
$$
uniformly on compact sets in probability as $n\to \infty$ except possibly on some interval compact $I$ in $(t_c, \infty)$,
where
\begin{equation}\label{Def:phi}
\varphi(t) = \frac{1- u(t)}K = \int_0^t 1-\theta(s)^2 ds.
\end{equation}
The proof is analogous but more complicated than that of Proposition \ref{bounds distance}. Note that if $\sigma$ is a permutation, every transposition can at most increase the number of cycles by 1. Hence if $\sigma$ has $N(\sigma)$ cycles, after one step $s\in \Gamma$, $\sigma$ has at most $N(\sigma)+ K$ cycles. Thus after $p$ steps, the number of cycles of $\sigma$ is at most $N(\sigma) + K p$. Since the identity permutation has exactly $n$ cycles, we conclude that
\begin{equation}\label{lbdistance}
d(\sigma) \ge \frac1K(n- N(\sigma)).
\end{equation}
Together with Lemma \ref{L:clusters2} and the definition of $\varphi(t)$, this proves the lower bound in Theorem \ref{3-cycle}.

Note that this bound would be sharp if we can find a path to the identity which makes a fragmentation at each step. We now work our way towards the upper-bound, which shows that indeed such a path may be found except that we may have to add an additional $o(n)$ coagulation steps. Call a component of $H_t$ \emph{good} if it is a hypertree and bad otherwise; a hyperedge is good if its component is good. Likewise, call a cycle $C$ of $\sigma(t)$ good if its associated component $\bar C$ in $H_t$ is a hypertree. Therefore, a good cycle is one which has never been involved in fragmentations, i.e., its history consists only of coagulation events. Fix $t>0$ and write $\sigma(tn) = \sigma^g \cdot \sigma^b$, where $\sigma^g$ is the product of all good cycles of $\sigma(tn)$ while $\sigma^b$ is the product of all bad cycles. Thus
$$
\sigma^g = c^g_1 \ldots c^g_{r(g)}, \ \ \sigma^b = c^b_1 \ldots c^b_{r(b)}
$$
Note that by (\ref{frag}), and recalling that there can never be more than $\sqrt{n}$ cycles greater or equal to $\sqrt{n}$, we have $r(b) \le n^{3/4}$ say, and the total mass of cycles in $\sigma^b$ is
\begin{equation}\label{massbad}
\frac{|\sigma^b|}n = \theta(t) + o(1),
\end{equation}
where $o(1)$ stands for a term that converges to 0 in probability, u.c.p. Assume for simplicity that $\g$ is an odd conjugacy class that generates all of $\cS_n$ (the arguments below can easily be adapted otherwise). To start with, note that in less than $o(n)$ moves, we can transform $\sigma(tn)$ into $\sigma'$ where all the cycles $c^b_1, \ldots, c^b_{r(b)}$ have been coagulated to form one large bad cycle, leaving the good cycles unchanged. Thus $\sigma'= \sigma^g\cdot \sigma'^b$, where $\sigma'^b$ has only one nontrivial cycle, whose size is $|\sigma^b|$. By the triangle inequality, it then suffices to find a path between $\sigma'$ and the identity of length approximately given by \eqref{lbdistance}.

Roughly speaking, our strategy for constructing a path between $\sigma'$ and the identity using steps from the conjugacy class $\g$ is to systematically destroy every good cycles as much as possible before destroying the bad cycles, as the good cycles are slightly harder to destroy than the bad cycles. Indeed, consider a cycle $C$ such that $|C| > \gg$. Then note that applying a judicious permutation $s \in \g$ to $C$ we can transform $C$ into $C'$ where the elements of $C \setminus C'$ are now fixed points, and $|C'| = |C| - K$.
Therefore, for an arbitrary cycle $C$, we get that
\begin{equation}\label{destroy}
C \text{ can be destroyed in at most $\frac{|C|}K + O(1)$ steps,}
\end{equation}
where the term $O(1)$ is nonrandom, uniformly bounded in $C$ and $n$. This bound is useful for the large bad cycle that makes up $\sigma'$, but does not help for small (good) cycles, of which there are of order $n$.

However, if $C$ is a good cycle and $e_1, \ldots, e_j$ are the hyperedges associated with the component of $C$ in $G(tn)$ (corresponding to the application of certain cycles as part of a step prior to time $tn$, say $\gamma_1, \ldots, \gamma_j$, which we will call the \emph{subcycles} of $C$), then $C$ can be destroyed by applying successively $j$ random cycles $\gamma'_1, \ldots, \gamma'_j$ of respective length $|e_1|, \ldots, |e_j|$, in some specified order. Unfortunately, it may not always be possible to perform exactly the sequence $\gamma'_1, \ldots, \gamma'_j$ as there are some arithmetic constraints on the sizes of the cycles that can be performed (indeed, each application of a cycle must be a part of the application of a permutation $s \in \g$). A problem may thus arise because, among good components the smaller hyperedges tend to be over-represented. This is made precise by the next lemma.

\begin{lemma}
 \label{L:edgecount} Fix $t>0$. Let $j\ge 2$ such that $k_j>0$. Then the number $ U_j(tn)$ of good hyperedges of size $j$ in $H_{tn}$, satisfies
 \begin{equation}\label{massj}
 \frac{U_j(tn)}{n} \to_p k_j t (1-\theta(t))^j
 \end{equation}
\end{lemma}

\begin{proof}
 The number of $j$-edges that have been added to $G(tn)$ is a Poisson random variable with mean $tnk_j$. For each such edge, the probability that it is not in the giant component $W$converges to $(1-\theta(t))^j$. [To see this, note that by Lemma \ref{L:giant}, it suffices to check that none of the $j$ points are in a cluster of size greater than $\beta \log n$ for $\beta>0$ large enough. This involves checking a neighbourhood of these $j$ points so that no more than $j \beta \log n$ vertices' connections are revealed. Since this is much smaller than the $n^{1/2}$ neighbourhood size of the birthday problem The probability that the exploration
  Thus
 $\E(U_j(tn)) \sim tn k_j (1-\theta)^j$. while if $e$ and $e'$ are two randomly chosen $j$-edges,
 $\P( e \subset W, e' \subset W)$ converges for the same reasons to $(1-\theta(t))^{2j}$, so that $\cov(\indic{e \subset W}, \indic{e' \subset W}) \to 0$. Thus the lemma follows from the second moment method.
\end{proof}

Recall that $J$ is the maximal size of a cycle for a permutation $s \in \g$, so that the subcycles of size $J$ are the most under-represented among good cycles. Consider the path that leads from $\sigma_J:=\sigma'$ to $\sigma_{J-1}$ in $d_{J} = U_J(tn)/k_J$ steps, where $\sigma_{J-1}$ is the permutation obtained by destroying from $\sigma_J$ all the subcycles of size $J$ from all good cycles and completing each step by removing $k_j$ subcycles of size $j$ for $2\le j \le J-1$ among good cycles. At this point we may write $\sigma_{J-1} = \sigma^g_{J-1} \cdot \sigma'^b_{J-1}$, where $\sigma^b_{J-1} = \sigma'^b$ (so the bad part is unchanged) and $\sigma^g_{J-1}$ is the same as $\sigma^g$ but all subcycles of size $J$ have been destroyed.

If $\g$ consists only of $k$-cycles, then the estimate (\ref{destroy}) with (\ref{massbad}) finishes the proof of the theorem in that case. Else,  we still call the cycles of $\sigma^g_{j-1}$ good, and note that they may still be decomposed in subcycles of size $j\le J-1$.
We similarly construct inductively $\sigma_{J-2}, \ldots, \sigma_{1}$, where $\sigma_{j-1}$ is obtained from $\sigma_j$ by removing from it all good subcycles of size $j$. Each time a step $s= c_1\ldots c_{L}$ is performed, where $L= \sum_{\ell=2}^J k_\ell$, we take $c_\ell$ from the good subcycles of $\sigma^g_j$ if $\ell \le j$, while we use for $c_\ell$ vertices from  $\sigma^b_j$. This construction is possible so long as $\sigma^b_j$ does not ``run out of mass". However, by Lemma \ref{L:edgecount}, for every $\eps>0$ with high probability the total mass that is required from bad cycles in this procedure is no more than
$$
M = \sum_{j=2}^J j k_j (1+\eps) t n [(1-\theta(t))^2 - (1-\theta(t))^j],
$$
since $J$ is uniformly bounded in $n$. Thus if
$$
M \le \theta(t) n
$$
which is the initial mass of bad cycles (i.e., the mass of $\sigma_J^b  = \sigma'^b$), then the upper-bound (and hence the result) follows from \eqref{destroy} and \eqref{massbad}. Indeed, in that case, we have constructed a path to the identity where the only coagulations are made when going from $\sigma(tn)$ to $\sigma'$ and potentially when finishing to destroy the bad cycle $\sigma^b_{1}$. In any case that accounts for no more than $o(n)$ such coagulations (with high probability). Referring to the remark under \eqref{lbdistance}, it follows that this path has a length of no more than
$$\frac1K (n - N(\sigma)) + o(n) = \varphi(t)n + o(n).$$
It thus remains solely to prove that $M < \theta n$ with high probability if $t>t_c$ is sufficiently close to $t_c$.
%But since $t_c = (\sum_{j=2}^J j(j-1)k_j)^{-1}$, this is possible if
%\begin{equation}\label{condmass}
%\sum_{j=2}^J jk_j[(1-\theta)^2 - (1- \theta)^j] < \theta \sum_{j=2}^J j(j-1)k_j .
%\end{equation}
However, using that $1- (1-x)^\alpha < \alpha x$ if $\alpha \ge 1$ and $0<x<1$, we see that for all $t> t_c$
\begin{equation}
\frac{M}{\theta n}  \le (1+\eps) t(1-\theta)^2 \sum_{j=2}^J k_j j(j-2)
\label{Mdom}\end{equation}
Thus it suffices to prove that the right-hand side is strictly smaller than 1 if $t$ is sufficiently close to $t_c$ or if $t$ is sufficiently large. When $t \to t_c = (\sum_{j=2}^J j(j-1)k_j)^{-1}$, then this is easily verified, at least provided that $\g$ does not consist solely of $2$-cycles, in which case the result is already known. In the case $t \to \infty$, this comes from the fact that there exists $c>0$ such that for $t$ large enough
$$
1- \theta(t) \le e^{-ct}.
$$
In turn, this follows from the fact that $\theta(t)$ is the survival probability of a Galton-Watson process where the offspring distribution is \eqref{progeny} and can thus be bounded below stochastically by a Poisson random variable with mean $t$. This finishes the proof of the result.

\begin{rmk}
Based on numerical methods in several particular cases, we expect that the inequality $M \le \theta n$ holds in general (i.e., for all $t\ge t_c$ and all conjugacy class $\Gamma$ of finite length). This would imply in particular that the limiting result for the behaviour of the distance $d(t)$ should hold for all $t\ge 0$. In fact, in all the examples that we have looked at, the function $M/\theta$ appears to be monotone decreasing for $t> t_c$.

However the upper-bound \eqref{Mdom} is too crude for this, as it can be shown that the right-hand side does not have to be monotone and is in fact strictly greater than 1 for some values of $t>t_c$ provided that $\Gamma$ contains cycles of size large enough.
\end{rmk}

\subsection*{Acknowledgements}
This paper owes much to the kind support of Oded Schramm, who encouraged me to write this result down during a visit I made to him in August 2008. I am very grateful to him and the Theory Group at Microsoft Research for their invitation during that time.

\end{document}